\documentclass[reqno]{amsart}

\usepackage{comment}
\usepackage{geometry}                % See geometry.pdf to learn the layout options. There are lots.
\usepackage{graphicx}
\usepackage{amssymb}
\usepackage[utf8]{inputenc}
\usepackage{color}

\usepackage{mathtools}
\mathtoolsset{showonlyrefs}

\newcommand{\R}{\mathbb{R}}
\newcommand{\N}{\mathbb{N}}
\newcommand{\E}{\mathcal{E}}
\newcommand{\Ms}{\mathcal{M}}
\newcommand{\Es}{\mathcal{E}_H}
\newcommand{\W}{\mathcal{W}}
\newcommand{\A}{\mathcal{A}}
\renewcommand{\S}{\mathcal{S}}
\newcommand{\Rn}{{\R^n}}
\newcommand{\Ha}{\mathcal{H}}
\newcommand{\LL}{\mathcal{L}}

\newcommand{\Chi}{\mathcal{X}}
\newcommand{\eps}{\varepsilon}
\newcommand{\Laplace}{\Delta}
\renewcommand{\phi}{\varphi}
\newcommand{\dx}{\mathrm{d}x}
\newcommand{\abs}[1]{\left| #1 \right|}

\newcommand{\dH}{\mathrm{d}\mathcal{H}}

\newcommand{\dist}{\operatorname{dist}}
\newcommand{\sdist}{\operatorname{sdist}}
\newcommand{\sgn}{\operatorname{sgn}}
\newcommand{\spt}{\operatorname{spt}}
\newcommand{\Urep}{{\mathcal{U}_{r(\eps)}}}

\newtheorem{theorem}{Theorem}
\newtheorem{lemma}[theorem]{Lemma}
\newtheorem{prop}[theorem]{Proposition}

\title[A phase field model for a constrained Willmore energy]{A phase field model for the optimization of the Willmore energy in the class of connected surfaces}

\author{Patrick W.~Dondl}
\address{Patrick W.~Dondl, Durham University,
Science Laboratories,
South Rd,
Durham DH1 3LE,
United Kingdom}
\email{patrick.dondl@durham.ac.uk}

\author{Luca Mugnai}
\address{Luca Mugnai, Max Planck Institute for Mathematics in the
  Sciences, Inselstr. 22, D-04103 Leipzig}
\email{mugnai@mis.mpg.de}

\author{Matthias R{\"o}ger}
\address{Matthias R{\"o}ger, Technische Universit\"at Dortmund,
Fakult{\"a}t f{\"u}r Mathematik,
Vogelpothsweg 87,
D-44227 Dortmund}
\email{matthias.roeger@tu-dortmund.de}

\date{\today}                                          % Activate to display a given date or no date

\begin{document}
\begin{abstract}
We consider the problem of minimizing the Willmore energy connected surfaces with prescribed surface area which are confined to a finite container. To this end, we approximate the surface by a phase field function $u$ taking values close to $+1$ on the inside of the surface and $-1$ on its outside. The confinement of the surface is now simply given by the domain of definition of $u$. A diffuse interface approximation for the area functional, as well as for the Willmore energy are well known. We address the topological constraint of connectedness by a nested minimization of two phase fields, the second one being used to identify connected components of the surface. In this article, we provide a proof of Gamma-convergence of our model to the sharp interface limit.
\end{abstract}
\maketitle

\section{Introduction}
In many applications structures can be described as (local) minimizer of suitable bending energies. The most prominent example is the variational characterization of shapes of biomembranes by the use of Helfrich-type functionals of the form
\begin{align}
  \Es(\Sigma) \,&=\, \int_\Sigma k_1 (H-H_0)^2\,d\Ha^{2}\,
  +\,\,\int_\Sigma  {k}_2 K\,d\Ha^{2}, \label{eq:E-sharp}
\end{align}
where $\Sigma$ denotes a surface in $\R^3$,
$H$ and $K$ its mean and Gaussian curvature, and where the bending moduli $k_1,k_2$ and the spontaneous curvature $H_0$ are in the simplest case constant. Under given constraints on enclosed volume and surface area minimizers in the class of sphere-type surfaces agree well with typical shapes of biological cells. In the case of zero spontaneous curvature and if the Gaussian curvature term is neglected, $\E_H$ reduces to the well-known Willmore functional.

Whereas the restriction to topological spheres is natural in many applications it is sometimes more reasonable to consider the class of orientable connected surfaces of arbitrary genus instead. For example the inner membrane of mitochondria cells shields the inside matrix from the outside but shows -- in contrast to old textbook illustrations -- a lot of of handle-like junctions~\cite{Mannella:2001gt} and therefore represents a higher genus surface. In this example another natural constraint comes into play, given by the confinement of the inner membrane to a `container' that is given by the outer membrane of the mitochondria.

This motivates to consider the following variational problem:\\[2mm]
\text{}\hfill\emph{
\begin{minipage}{0.9\textwidth}
	Minimize $\Es$ in the class of all compact, connected, orientable surfaces without boundary that are embedded in $\Omega\subset \R^3$ and have prescribed surface area $S$.
\end{minipage} 
}
\\[2mm]
Such optimization problems for curvature energies are in general difficult and only a few rigorous results are available. Simon \cite{Simo93} proves the existence of minimizers for the Willmore energy for tori in $\R^3$. This result was extended  to surfaces with arbitrary prescribed genus \cite{BaKu03} and to surfaces of sphere-type with prescribed isoperimetric ratio \cite{Schy11}. 

Instead of solving the problem in its original formulation we here propose an approximation by a phase field type energy. Such a formulation is very well suited for numerical investigations and has been used extensively in similar problems (\textit{e.g.} \cite{DuWang04, DuWang05,  DuWang07, CampHern, CampHern2}). In such an approach the confinement condition is easily imposed and approximations of the Willmore and Helfrich energy are well-known. The key challenge is therefore to appropriately translate the condition that the surfaces (in the `sharp interface formulation') should be connected. We take care of this constraint by a nested minimization, one that is able to detect if multiple separated components of the boundary exist.

Our approach builds on an previous article by the authors where we have considered the confined Willmore minimization under a topological constraint in two dimensions\cite{DoMR11}. A similar idea of a nested minimization was already proposed in that paper. However, the approach presented here is easier to implement and can be applied in the physically relevant case of three dimensions. An interesting, related approach is presented in~\cite{Alexandr05}, where the authors propose using a logarithmic barrier method in a level set approach in order to prevent topological changes. A method of tracking topological changes was introduced in~\cite{Du:2007ux}, where the authors introduce a diffuse analogue to the Euler number and consider its change with time. This approach could be fruitfully combined with ours in order to study the topological transitions of connected and constrained  surfaces undergoing a gradient flow with respect to the Willmore energy.

\section{The sharp interface minimization problem}
Consider $n=2$ or $n=3$ and let $\Omega\subset \Rn$ be a given open bounded set with smooth boundary. Denote by $\Ms$ the set of all compact, connected, orientable $C^2$ surfaces without boundary that are embedded in $\Omega$ and have prescribed surface area $S$, where $S>0$ is a given constant. We associate to $\Sigma\in\Ms$ the enclosed inner region $U_\Sigma\subset\R^n$ and denote by ${H}$ the mean curvature vector of $\Sigma$.

In the following we will consider the Willmore energy
\begin{align}
  \W(\Sigma) \,&:=\, \int_\Sigma |H|^2\,d\Ha^{n-1}, \label{eq:def-W}
\end{align}
which is a special case of the general Helfrich-type energy $\Es$. Our interest is in the minimization of the Willmore energy in the class $\Ms$. Optimal structures, i.e. limit points (in a suitable sense) of minimizing sequences, can be expected to have touching points with the boundary $\partial\Omega$ of the container or points of self-contact. Therefore such structures are not necessarily embedded surfaces and do in general not belong to $\Ms$. The following proposition gives some information about limit points of sequences in $\Ms$ that are uniformly bounded in Willmore energy.
\begin{prop}\label{prop:min-seq}
Consider a sequence $(\Sigma_k)_{k\in\N}$ in $\Ms$, let $E_k$ be the associated inner sets, and assume 
\begin{align}
	\sup_{k\in\N} \W(\Sigma_k) \,<\, \infty. \label{eq:bound}
\end{align}
Then there exists a subsequence $k\to\infty$, a set $E\subset \bar{\Omega}$, and a Radon measure $\mu$ on $\Rn$ such that
\begin{align}
	\Chi_{E_k} \,\to\, \Chi_E \quad\text{ in }L^1(\Rn), \label{eq:min-seq1}\\
	\Ha^{n-1}\lfloor \Sigma_k\,\to\, \mu \quad\text{ as  varifolds}, \label{eq:min-seq2}\\
	\Chi_{E} \,\in\, BV(\Rn), \label{eq:min-seq1.1}\\
	\mu(\Rn) \,=\, S, \label{eq:min-seq3}\\
	\mu \,\geq\, |\nabla\Chi_E|, \label{eq:min-seq4}\\
	\mu \text{ is an $(n-1)$-dimensional integral varifold} \nonumber \\
	\text{with weak mean curvature }\vec{H}\in L^2(\mu)\text{ and multiplicity $\theta(\mu,\cdot)$}.%,\\ 
%	\mu_{\lfloor\{x\in\R^n:~\theta(\mu,x)\text{ is odd}\}}=\theta(\mu,\cdot)\Ha^{n-1}_{\lfloor\partial^*E}.
\end{align}
Moreover we have
\begin{align}
	\int |\vec{H}|^2 \,d\mu \,\leq\, \liminf_{k\to\infty} \W(\Sigma_k). \label{eq:min-seq5}
\end{align}
We define $\Ms_0$ to be the class of all pairs $(E,\mu)$ of sets in $E\subset \overline{\Omega}$ and Radon measures $\mu$ on $\Rn$ such that there exists a sequence $(\Sigma_k)_{k\in\N}$ in $\Ms$ with \eqref{eq:min-seq1}, \eqref{eq:min-seq2}.
\end{prop}
\begin{proof}
The conclusions follow from the BV-compactness Theorem and Allard's compactness Theorem.
\end{proof}
Below we will prove the existence of a recovery sequence of diffuse approximations for the smaller class of limit structures that can be approximated \emph{in energy}, more precisely the class
$\Ms_1$ of all pairs $(E,\mu)\in \Ms_0$ such that  there exists a sequence $(\Sigma_k)_{k\in\N}$ in $\Ms$ with \eqref{eq:min-seq1}-\eqref{eq:min-seq2} and
\begin{align}
	\W(\mu) := \int |\vec{H}|^2 \,d\mu \,=\, \lim_{k\to\infty} \W(\Sigma_k). \label{eq:app-energy}
\end{align}
In space dimension $n=2$ the classes $\Ms_0$ and $\Ms_1$ in fact coincide, see \cite[Proposition 2.2]{DoMR11}. Whether such a property also holds in three space dimensions (at least for limits of minimizing sequences) is at present an open question, that we do not address here. 
\section{Diffuse interface approximation and main results}
In this section we formulate the phase field approximation of our minimization problem and introduce our main results. In the following we consider $n\in \{2,3\}$ and an open, bounded set $\Omega \subset \R^n$ with smooth boundary that represents the confinement condition. We are interested in an approximation of our minimization problem in the class of smooth phase fields $u\colon \Rn \to \R$ that satisfy a boundary condition
\begin{align}
	u\,=\, -1 \quad\text{ on }\Rn\setminus\Omega \label{eq:pf-conf}
\end{align}
(alternatively one could prescribe clamped boundary conditions on $\partial\Omega$).
For $\eps>0$ we choose as an approximation of the area functional 
the well-known Modica-Mortola functional
\begin{align}
	\S_\eps(u) := \frac{1}{c_0}\int_\Omega \Big(\frac{\eps}{2} \abs{\nabla u}^2 + \frac{1}{\eps}W(u)\Big) \,\dx,
\end{align}
where $W$ is a suitable double well potential that we here chose as $W(r)=\frac{1}{4}(1-r^2)^2$ with the scaling factor $c_0:= \int_{-1}^1\sqrt{2W(s)}\,\mathrm{d}s$ fixed such that $\S_\eps$ Gamma-converges for $\eps\to 0$ to the area functional, see \cite{Modi87}. To a phase field $u$ we associate the diffuse interface measures
\begin{gather}
  \mu_\eps \,:=\, \frac{1}{c_0}\Big(\frac{\eps}{2}|\nabla u_\eps|^2 + \frac{1}{\eps}W(u_\eps)\Big)\,\dx \label{eq:def-mu}
\end{gather}

We will restrict ourselves in the following to the Willmore functional as the relevant bending energy and consider the following diffuse analogue
\begin{gather}
	\W_\eps(u)\,:=\,  \frac{1}{c_0}\int_{\Omega} \frac{1}{\eps} \left( \eps \Laplace u - \frac{1}{\eps} W'(u) \right)^2 \,\dx. \label{eq:def-Weps}
\end{gather}
In space dimensions $n=2,3$ the energy $\W_\eps$ Gamma-converges at regular limit points to the Willmore energy \cite{BeMu05,RoeSc06}.\\
To detect if multiple connected components are present we define for given $u\in W^{2,2}(\Omega)$ a functional
$\A_{u,\eps}:W^{1,2}(\Omega)\,\to\, \R$,
\begin{align}
	\A_{u,\eps} (\phi) := \abs{\int_\Omega  \frac{1}{\eps} \widetilde{W}(u) \phi \,\dx} + 
         \int_\Omega  \left( \frac{9}{\eps^{\frac{3}{2}}} \widetilde{W}(\overline{\lambda}u) + \eps \right) W(\phi) \,\dx +
\int_\Omega \left(\frac{8}{\eps^{\frac{3}{2}}} \widetilde{W}(\overline{\lambda}u)+\eps\right)  \abs{\nabla \phi}^2 \,\dx.\label{eq:def-Aeps}
\end{align}
Here the function $\widetilde{W}\in C_c^\infty(\R)$ and $\overline{\lambda}$ are chosen such that
\begin{align}
	\widetilde{W}(s) = \left\{\begin{array}{ll} 
		1 & \textrm{if $\abs{s}\le 1-\lambda$} \\
		0 & \textrm{if $\abs{s}>\overline{\lambda}^{-1}(1-\lambda)$} \\ 
		\textrm{monotone } & \textrm{otherwise}
	\end{array}\right.
\end{align}
for $0<\lambda<1$ sufficiently small\footnote{As shown in the proof of Theorem~\ref{thm:lower}, a possible choice for $\lambda$ is $\frac{1}{61}$.} and $\overline{\lambda} < 1$ such that $\overline{\lambda}^{-1}(1-\lambda) < 1$. 
Note that the constants above have been chosen for notational convenience -- the important aspect in the functional above is the scaling relationship between the individual terms.

The idea is to penalize the difference between the infimum of $\A_{u,\eps}$ and the value of $\A_{u,\eps}$ for $\phi=1$. The reasoning behind is that if two isolated components are given, then one could choose $\phi = 1$ and $\phi=-1$ in suitable neighborhoods of the components, achieving a smaller value than for the choice $\phi=1$.  In contrast, if the diffuse interface only consists of one component any gain achieved in the first term of $\A_{u,\eps}$ will be expensive by the combination of the second and third term in \eqref{eq:def-Aeps}. Finally the additional $\eps$-term in the penalization of the gradient ensures that $\A_{u,\eps}$ is coercive in $W^{1,2}(\Omega)$, which is convenient for numerical implementation.

Putting together all contributions we finally obtain the total energy functional.
Let $S>0$ be fixed, take $\eps>0$ and $0<\sigma<1$. We then define for $u\in W^{2,2}(\Omega)$
\begin{align}
	\E_\eps(u) \,:=\, &\W_\eps(u) + \frac{1}{\eps^{1-\sigma}} \left( \S_\eps(u) - S \right)^2 + \notag\\
	&+\frac{1}{\eps^{1/2-\sigma/2}} \left( \inf_{\phi\in W^{1,2}(\Omega)} \A_{u,\eps}(\phi)- \int_\Omega \frac{1}{\eps} \widetilde{W}(u) \right)^2. \label{eq:def-Eeps}
\end{align}
We then consider the minimization of $\E_\eps(u)$ for $u\in W^{2,2}(\Omega)$ subject to the boundary conditions~\eqref{eq:pf-conf}.

From a practical point of view, the two-field minimization problem of course yields some difficulties for implementation. It is necessary that an \emph{absolute} minimizer $\varphi$ for the non-convex functional $\A_{u,\eps}$ is found---and, at least for some configurations, we expect a non-separating $\phi$ to still be a local minimizer. One possible way of overcoming such a problem could be by a method similar to the one introduced by~\cite{Bourdin:2007jz} (who are also dealing with a model incoroporating two-phase fields). In our case, we would add a growing coefficient in front of the first term of $\A_{u,\eps}$ until the uniform (i.e., non-separating) $\phi$ loses stability as a minimizer and then use the resulting $\phi$ as a starting point to find the minimizer for the original $\A_{u,\eps}$. An implementation is work in progress.

The main contribution of this paper is the justification of the approximation property in form of a Gamma convergence result for $\E_\eps$.
The major difficulty thereby is the treatment of the  $\A_{u,\eps}$ term.

We first show that the respective functionals admit minimizers.
\begin{prop}
The functional $\A_{u,\eps}$  admits, for any
$u\in W^{2,2}(\Omega)$, a minimizer in $W^{1,2}(\Omega)$.
The functional $\E_\eps$ admits a minimizer in the class of functions $W^{2,2}(\Omega)$ subject to boundary condition~\eqref{eq:pf-conf}.
\end{prop}
\begin{proof}
For fixed $u\in W^{2,2}(\Omega)$, the direct method of the calculus of variations yields existence of a minimizer for the functional $\A_{u,\eps}$. We claim that the function
\begin{align*}
A \colon C^0 &\to \R \\
u &\mapsto  \inf_{\phi\in W^{1,2}(\Omega)} \A_{u,\eps}(\phi)
\end{align*}
is continuous. In oder to see this, consider $u_1, u_2 \in W^{2,2} \subset C^0$ and assume without loss of generality that $\inf_{\phi\in W^{1,2}(\Omega)} \A_{u_1,\eps}(\phi) \le \inf_{\phi\in W^{1,2}(\Omega)} \A_{u_2,\eps}(\phi)$. Denote by $\phi_1$ the minimizer of  $\A_{u_1,\eps}$. Plugging $\phi_1$ into $\A_{u_2,\eps}$ yields
$$
\A_{u_2,\eps}(\phi_1) \le \inf_{\phi\in W^{1,2}(\Omega)} \A_{u_1,\eps}(\phi) + C(\eps, \Omega)||u_1-u_2||_\infty \left( || \phi_1 ||_{L^1} + || W(\phi_1) ||_{L^1} + || |\nabla(\phi_1)|^2 ||_{L^1} \right).
$$
Noting that all the $L^1$-norms on the right hand side are bounded uniformly for any minimizer $\phi_1$ of $\A_{u_1,\eps}$ regardless of $u_1$\footnote{To see this, compare with $\phi \equiv 0$ as a test function to get a uniform upper bound for  $\inf_{\phi\in W^{1,2}(\Omega)} \A_{u_1,\eps}(\phi)$. This yields a bound for $||\eps W(\phi||_{L^1}$ and  $||\eps \nabla \phi||_{L^1}$. A bound for $|| \phi||_{L^1}$ follow immediately.} by a constant depending also only on $\eps$ and $\Omega$, continuity, in fact even Lipschitz continuity follows.

Using again the direct method of the calculus of variations for the functional $\E_\eps$ we see that a subsequence of a minimizing sequence converges weakly in $W^{2,2}$, thus one can extract another subsequence which strongly converges in $C^{0,\alpha}$ for some $\alpha>0$. This (together with another  compact embedding in $W^{1,p}$ for $p<6$) ensures lower semicontinuity of all lower order terms in the functional.
\end{proof}

Our main results are a lower and upper bound for the above diffuse approximations of our minimization problem. 
\begin{theorem}[Upper bound]
\label{thm:main}
Let an arbitrary $(E,\mu)\in \Ms_1$ be given and set $u := 2\chi_E-1$. Then there exists a sequence of smooth phase fields $u_\eps: \Rn\to\R$ with \eqref{eq:pf-conf} such that
\begin{align}
	u_\eps \,&\to\, u\qquad\text{ in }L^1(\Rn), \label{eq:limsup1}\\
	\mu_\eps \,&\to\, \mu \qquad\text{ as Radon measures} \label{eq:limsup2}
\end{align}
and such that
\begin{align}
	 \lim_{\eps \to 0} \E_\eps(u_\eps) \,=\, \W(\mu). \label{eq:limsup3}
\end{align}
\end{theorem}
\begin{theorem}[Lower bound]
\label{thm:lower}
Consider a sequence $(u_\eps)_{\eps>0}$ of phase fields $u_\eps:\Rn\to\R$ with \eqref{eq:pf-conf} and
\begin{align}
	&\liminf_{\eps\to 0} \E_\eps(u_\eps)\,<\, \infty, \label{eq:liminf3}\\
	&u_\eps \,\to\, u \qquad\text{ in }L^1(\Rn), \label{eq:liminf1}\\
	&\mu_\eps \,\to\, \mu \qquad\text{ as Radon measures} \label{eq:liminf2}
\end{align}
for a function $u\in L^1(\Rn)$ and a Radon measure $\mu$ on $\Rn$. Then the following properties hold.\\
\begin{align}
	\text{There exists a set } E\subset \overline{\Omega}\text{ of finite perimeter with }u=2\Chi_E-1, \label{eq:liminf4}\\
	\mu(\Rn) \,=\, S, \label{eq:liminf5}\\
	\mu \,\geq\, |\nabla u|, \label{eq:liminf6}\\
	\mu \text{ is an $(n-1)$-dimensional integral varifold} \nonumber \\
	\text{with weak mean curvature }\vec{H}\in L^2(\mu)\text{ and multiplicity $\theta(\mu,\cdot)$}\label{eq:liminf7}.%\\
%	\mu_{\lfloor\{x\in\R^n:~\theta(\mu,x)\text{ is odd}\}}=\theta(\mu,\cdot)\Ha^{n-1}_{\lfloor\partial^*E}.\label{eq:liminf7bis}
\end{align}
Finally, $\mu$ represents a connected structure in the sense that there are no two open sets $\Omega_1,\Omega_2\subset \Rn$ with disjoint closure such that
\begin{align}
	\mu(\Omega_i) \,>\, 0 \,(i=1,2)\quad\text{ and } \mu\big(\Rn\setminus(\Omega_1\cup\Omega_2)\big)\,=\,0 \label{eq:liminf8}
\end{align}
and we have the lower estimate
\begin{align}
	\W(\mu) \,\leq\, \liminf_{\eps\to 0} \E_\eps(u_\eps). \label{eq:liminf9}
\end{align}
\end{theorem}
We remark that also the corresponding compactness property holds: For an arbitrary sequence $(u_\eps)_{\eps>0}$ of smooth phase fields with $\liminf_{\eps\to 0} \E_\eps(u_\eps)<\infty$ there exist a subsequence $\eps\to 0$,  $u\in L^1(\Rn)$, and a Radon measure $\mu$ on $\Rn$ with \eqref{eq:liminf1}, \eqref{eq:liminf2}.
\section{Construction of a recovery Sequence.}
In this section, we show that, given $u = 2\chi_E-1$ for $(E,\mu) \in \Ms_1$ as in the statement of Theorem~\ref{thm:main}, there exists a sequence $u_\eps \to u$ in $L^1$, $|\nabla u| \to \mu$ as a Radon measure such that $\limsup_{\eps \to 0} \E_\eps(u_\eps) \le \W(\mu)$. 
Note that we can, in view of the definition of $\Ms_1$, assume $\partial E$ to be embedded and $C^2$, and $\mu = |\nabla \chi_E|$. A diagonal sequence approach will yield the desired result otherwise.
We approximate $u$ by a common optimal profile construction $u_\eps$, given as follows. 

Since $\partial E$ is now an embedded $C^2$ surface, there exists $\delta>0$ such that the signed distance function
$d := \sdist(\cdot, \partial E)$ (taken positive on the inside of $E$) is of class $C^2$ on $\{ x\in\Omega: \dist(x,\partial E) \le \delta \}$.
As in \cite{DoMR11} we follow the standard construction and consider the optimal one-dimensional profile $q:\R\to (-1,1)$,
\begin{gather}
  -q^{\prime\prime} + W^\prime(q)\,=\, 0, \label{eq:opro1}\\
  q(-\infty)\,=\, -1, \quad q(+\infty)\,=\,1, \quad q(0)\,=\,0
  \label{eq:opro2}
\end{gather}
and note that 
\begin{gather}
  q'(r)\,=\,\sqrt{2W(q(r))} \label{eq:equipart-q}
\end{gather}
holds for all $r>0$.
Next fix  a smooth symmetric cut-off function $\eta\in C^\infty(\R)$,
\begin{gather*}
  0\leq \eta\leq 1,\quad \eta(r) = 1\text{ for }r\in[-1,1],\quad 
  \eta(r) = 0 \text{ for }|r|\geq 2,\quad \eta'\,\leq\,0.
\end{gather*}
For $\delta>0$ as above we then define
\begin{gather*}
  q_\eps(r)\,:=\,
  \eta\big(\frac{2r}{\delta}\big)q\big(\frac{r}{\eps}\big) +
  \sgn(r)\big(1-\eta\big(\frac{2r}{\delta}\big) \big) 
\end{gather*}
and claim that
\begin{gather}
  u_\eps(x)\,:=\,  q_\eps(d(x)). \label{eq:def-recovery}
\end{gather}
defines a suitable recovery sequence $(u_\eps)_{\eps>0}$.

The choice of $u_\eps(x)$ immediately ensures convergence of the Willmore energy $\W_\eps$, see for example \cite{BePa93,BeMu05}. It remains to show that the additional terms in $\E_\eps(u_\eps)$ vanish in the limit $\eps\to 0$.

With this aim we consider the following parametrization of the $\delta$-strip around $\partial E$
\begin{align}
 \psi: \partial E \times (-\delta,\delta)\,\to\, \Omega,\quad \psi(z,t)\,=\, z+t\nu(z),
\end{align}
where $\nu$ denotes the inner unit normal of $E$. By the choice of $\delta$ we find that $\psi$ is injective and continuously differentiable. Its Jacobi-determinant can be estimated by
\begin{align}
\label{eq:jac_est}
	\abs{\det J\psi(z,t) - 1} \le C t
\end{align}
with a constant $C$ depending only on the second fundamental form of $\partial E$. We now calculate, assuming
$r:=\eps^{1-\sigma/2} <\delta/2$,
\begin{align}
  c_0 S_\eps(u_\eps)\,&=\, \int_{\partial E} \int_{-\delta}^{\delta}
  \Big(\frac{\eps}{2} q_\eps'(t)^2 + \frac{1}{\eps}W(q_\eps(t)) \Big)
  \det J\psi(z,t) \,dt\,\dH^{n-1}(z) \notag\\
  &=\, \int_{\partial E}  \int_{-r}^{r}  \frac{1}{\eps}\Big(\frac{1}{2}
  q'(t/\eps)^2 + W(q(t/\eps)) \Big)\det J\psi(z,t)\,dt\,\dH^{n-1}(z)\notag\\
  &\quad + \int_{\partial E}  \int_{\{r<|t|<\delta\}}  \Big(\frac{\eps}{2}
  q_\eps'(t)^2 + 
  \frac{1}{\eps}W(q_\eps(t)) \Big) \det J\psi(z,t)\,dt \,\dH^{n-1}(z). \label{eq:est-L0}
\end{align}
The second term vanishes faster than any power of $\eps$ (see \cite{DoMR11}), and so does
\begin{align}
\abs{\int_{\partial E}  \int_{-r}^{r}  \frac{1}{\eps}\Big(\frac{1}{2}
  q'(t/\eps)^2 + W(q(t/\eps)) \Big)\,dt\,\dH^{n-1}(z)  - c_0 S}.
\end{align}
Using~\eqref{eq:jac_est} for the first term we deduce
\begin{align}
  &\Big|\int_{\partial E}  \int_{-r}^{r}  \frac{1}{\eps}\Big(\frac{1}{2}
  q'(t/\eps)^2 + W(q(t/\eps)) \Big)\det J\psi(z,t)\,\mathrm{d}t\,\dH^{n-1}(z)\\
  &\qquad \qquad \quad -\int_{\partial E}  \int_{-r}^{r}  \frac{1}{\eps}\Big(\frac{1}{2}
  q'(t/\eps)^2 + W(q(t/\eps)) \Big) \,\mathrm{d}t \,\dH^{n-1}(z)\Big|\\
\leq\, &\int_{\partial E}  \int_{-r}^{r}  \frac{1}{\eps}\Big(\frac{1}{2}
  q'(t/\eps)^2 + W(q(t/\eps)) \Big)C|t|\,\mathrm{d}t\,\dH^{n-1}(z)\\
  \leq\,& rC\int_{\partial E}  \,\dH^{n-1}(z) \,\leq\, C\eps^{1-\sigma/2},
\end{align}
which implies
\begin{align}
	\abs{ S_\eps(u_\eps) - S } \le C\,\eps^{1-\sigma/2}.  \label{eq:limsup-area}
\end{align}
This shows that the surface area-penalty term, the second term in the energy $\E_\eps$, vanishes for $\eps \to 0$.

We now turn our attention to the third term. The choice $\phi \equiv +1$ as a test function shows immediately that 
\begin{align}
 \inf_{\phi\in W^{1,2}(\Omega)} \A_{u_\eps,\eps}(\phi) \le \int_\Omega  \frac{1}{\eps} \widetilde{W}(u_\eps). \label{eq:A_upper_bd}% \le C  \int_\Omega  \frac{1}{\eps} {W}(u_\eps) \,\leq\, C S. 
 \end{align}

We now switch to a modified functional in $\phi$, where we only allow the argument to take the values $\pm 1$, making it easier to prove a lower bound on the infimum. We will afterwards estimate the original functional by the modified one from below. Furthermore, we restrict our attention to the neighborhood of the transition layer where $\widetilde{W}(u_\eps) > 0$. Note that
\begin{gather*}
	\{x\in\Omega:~\widetilde{W}(u_\eps(x))>0\}=\{x\in\R^n:~\dist(x,\partial E)<r(\eps)\},
	\\
\text{ where }r(\eps):=\eps q^{-1}(1-\lambda), \text{ that is } r(\eps)=\eps r(1).
\end{gather*}
and define
\begin{align}
\label{eq:Urep}
\mathcal{U}_r := \{x\in\R^n:~\dist(x,\partial E)<r\},
\end{align}
in particular,  $\{\widetilde{W}(u_\eps) > 0\}= \mathcal{U}_{r(\eps)}$. Let now $\tilde{\A}_{u_\eps,\eps} \colon BV(\Urep;\,\{-1,+1\}) \to \R$ be defined by
\begin{align}
\tilde{\A}_{u,\eps}(\phi) = \abs{\int_\Urep \frac{1}{\eps} \widetilde{W}(u)\,\phi\, \dx} + 
\int_\Urep \frac{1}{\eps^{\frac{3}{2}}} \widetilde{W}(\overline{\lambda} u) \,d\!\abs{\nabla \phi}.
\end{align}
Note that for $\phi \equiv 1$ or  $\phi \equiv -1$ the modified functional agrees with the original one.

We thus first need to prove the following estimate on the modified functional, stating that for sufficiently small $\eps$, the minimizer is trivial. 
\begin{prop}\label{prop:dr}
There exists $\eps_0:=\eps_0(n,S,\partial E)$ such that for every $\eps<\eps_0$ we have
\begin{align}
	\inf_{\phi\in BV(\Urep;\,\{-1,+1\})} \tilde{A}_{u_\eps,\eps}(\phi) \,=\, \tilde{A}_{u_\eps,\eps}(1) \,=\, \tilde{c}_0 S+O(\eps^2), \label{eq:prop3}
\end{align}
for a $\tilde{c}_0>0$ independent of $\eps$.
\end{prop}

In order to prove Proposition \ref{prop:dr} we need the following Poincar{\'e}-type estimate, which will be used to link the gain in the first term of the functional achieved by having $\phi$ switch signs to a loss in the second term.
\begin{lemma}\label{lem:balotelli} Let $E\subset\R^n$ be open and such that $\partial E$ is a connected, compact $C^2$ hypersurface. For fixed $0<\delta<1/2\overline{\kappa}$, where 
$
\overline{\kappa}:=\max\{\sum_{i=1}^{n-1}\vert\kappa_i(z)\vert:~z\in\partial E\},
$
with principle curvatures $\kappa_i$, consider, for $r\in (0,\delta)$,  
$$
\mathcal U_r:=\{x\in\R^n:~\dist(x,\partial E)<r\}
$$
as above. There exists $C:=C(n,\partial E)$ such that for every $r\in (0,\delta)$ and every $A\subset\R^n$ with $\chi_A\in BV_{loc}(\R^n)$, we have
$$
 \min\big\{\|\chi_A\|_{L^1(\mathcal U_r)},\|1-\chi_A\|_{L^1(\mathcal U_r)}\big\}\leq \left(\frac{\delta}{r}\right)^{\frac{1}{n-1}} C \left(\int_{\mathcal U_r}d\vert\nabla\chi_A\vert \right)^{n/n-1}. 
$$
\end{lemma}

\proof
Denote by: 
\begin{itemize}
\item $\sdist(\cdot,\partial E)$ the signed distance from $\partial E$ positive inside $E$; 
\item $\Pi_{\partial E}$ the projection on $\partial E$;
\item $\nu(\cdot)$ the unit outward normal to $\partial E$,
\item $\{\tau_i(\cdot)\}_{i=1}^{n-1}$ an ortho-normal basis for the tangent space of $\partial E$ made of principle directions.
\end{itemize}
By the smoothness assumption on $\partial E$ and by the choice of $\delta$,  for every $r\in (0,\delta)$ the map
$$
\Phi_r:\mathcal U_\delta\,\to\,\mathcal U_r,\qquad \Phi_r(x)=\Pi_{\partial E}(x)-\frac{r}{\delta}\sdist(x,\partial E)\nu(\Pi_{\partial E}(x)),
$$
is $C^1$. Moreover we have~\cite[Chapter 14.6]{GilTru}
\begin{align*}
& D\Phi_r(x)\big[\nu(\Pi_{\partial E}(x))]=\frac{r}{\delta}\nu(\Pi_{\partial E}(x)),
\\
& D\Phi_r(x) \big[\tau_i(\Pi_{\partial E}(x))\big]=\left(1-\frac{r}{\delta}\sdist(x,\partial E)\kappa_i(\Pi_{\partial E}(x))\right)\tau_i(\Pi_{\partial E}(x)).
\end{align*}
Since $\{\nu(\Pi_{\partial E}(x)),\tau_1(\Pi_{\partial E}(x)),\dots,\tau_{n-1}(\Pi_{\partial E}(x))\}$ form an ortho-normal basis of $\R^n$, and since by the choice of $\delta$ and $x\in\mathcal U_\delta$ we have $\frac{1}{2}<1-\frac{r}{\delta}\sdist(x,\partial E)\kappa_i(\Pi_{\partial E}(x))<\frac{3}{2}$, we can conclude that
\begin{gather}
	\frac{r}{\delta}3^{1-n}<\det D\Phi_r(x)<\frac{r}{\delta}3^{n-1},
	\quad
	\frac{r}{\delta}3^{2-n}<J_{n-1}\Phi_r(x), \label{eq:est-jac}
\end{gather}
where $J_{n-1}$ is the $n-1$-dimenstional Jacobian with respect to $\partial E$.
Next choose an arbitrary $A\subset\R^n$ such that $\chi_A\in BV_{loc}(\R^n)$. Without loss of generality we may assume that
\begin{gather*}
	\|\chi_A\|_{L^1(\mathcal U_r)} \,\leq\, \|1-\chi_A\|_{L^1(\mathcal U_r)}.
\end{gather*}
By the two inequalities above and the area formula we get
\begin{gather*}
\|\chi_A\|_{L^1(\mathcal U_r)}=\int_{\Phi_r^{-1}(A)\cap\mathcal U_\delta}\det D\Phi_r\,dx\leq
3^{n-1}\frac{\delta}{r}\|\chi_{\Phi_r^{-1}(A)}\|_{L^1(\mathcal U_\delta)},
\\
\vert\nabla\chi_A\vert(\mathcal U_r)\geq \frac{\delta}{r}3^{2-n}\vert\nabla\chi_{\Phi_r^{-1}(A)}\vert(\mathcal U_\delta).
\end{gather*}
Moreover by the Poincar\'e-Wirtinger inequality (see for example \cite{Attouch:2006vj} in the proof of Lemma 10.3.2.) we can find $C':=C'(\delta,\partial E,n)>0$ such that, for every $\widetilde A\subset\R^n$ with $\chi_{\widetilde A}\in BV_{loc}(\R^n)$, we have
$$
 \|\chi_{\widetilde A}\|_{L^1(\mathcal U_\delta)}\leq   2\left\|\chi_{\widetilde A} - \frac{1}{\mathcal{L}(\mathcal U_\delta)}\int_{\mathcal U_\delta}\chi_{\widetilde A}\right\|_{L^1(\mathcal U_\delta)}\leq C' \left(\int_{\mathcal U_\delta}d\vert\nabla\chi_{\widetilde A}\vert \right)^{n/n-1}.
$$
Hence
\begin{align*}
\|\chi_A\|_{L^1(\mathcal U_r)}\leq & 3^{n-1}\frac{\delta}{r}\|\chi_{\Phi_r^{-1}(A)}\|_{L^1(\mathcal U_\delta)} \leq 3^{n-1}\frac{\delta}{r} C'\big(\vert\nabla\chi_{\Phi_r^{-1}(A)}\vert(\mathcal U_\delta)\big)^{n/n-1}
\\
\leq &  3^{n-1}\frac{\delta}{r} \Big(\frac{r}{\delta}3^{n-2}\Big)^{n/n-1} C'\big(\vert\nabla\chi_A\vert(\mathcal U_r)\big)^{n/n-1}.
\end{align*}
\endproof
We need another Lemma.
\begin{lemma}\label{lem:compl}
Let $u_\eps:\Omega\to\R$, $\eps>0$ be the family of functions defined in \eqref{eq:def-recovery} and $\Urep$ as in~\eqref{eq:Urep}. Then there exists a constant $C>0$ that only depends on $E,q,\widetilde{W}$, such that for all 
$\varphi\in L^1(\Urep;\{-1,1\})$ with
\begin{align}
	\int_{\Urep} \widetilde{W}(u_\eps)\varphi\,dx \,\geq\,0 \label{eq:L5-cond}
\end{align}
the estimate
\begin{align}
	\big|\{\varphi<0\}\cap \Urep \big| \,&\leq\, C \,\big|\{\varphi>0\}\cap \Urep \big| \label{eq:L5}
\end{align}
holds.
\end{lemma}
\begin{proof}
Using the transformation $\Phi_r$ from Lemma \ref{lem:balotelli} with $r=r(\eps)=\eps r(1)$ we obtain from \eqref{eq:L5-cond}
\begin{align}
	0\,&\leq\, \int_\Urep \widetilde{W}(u_\eps)\varphi\,dx \,=\, \int_{U_{r(\eps)}} W(q(\frac{d}{\eps}))\varphi\,dx \notag\\
	&=\, \int_{U_\delta} \widetilde{W}(q(\frac{r(1)d}{\delta}))\varphi\circ\Phi_r \det D\Phi_r \notag\\
	&\leq\, \frac{\eps r(1)}{\delta} \int_{U_\delta} \widetilde{W}(q(\frac{r(1)d}{\delta}))\tilde\varphi\,dx, \label{eq:L5-1}
\end{align}
where we have used \eqref{eq:est-jac} and where we have defined $\tilde\varphi:= -3^{1-n}\Chi_{\{\varphi\circ\Phi_r<0\}} + 3^{n-1}\Chi_{\{\varphi\circ\Phi_r>0\}}$.

Clearly we have $\{\tilde{\varphi}>0\}=\{\varphi\circ\Phi_r>0\}$ and $\{\tilde{\varphi}<0\}=\{\varphi\circ\Phi_r<0\}$. Using once more the transformation formula and \eqref{eq:est-jac} we deduce
\begin{align}
	\frac{\big|\{\varphi<0\}\cap \{\widetilde{W}(u_\eps)>0\}\big|}{\big|\{\varphi>0\}\cap \{\widetilde{W}(u_\eps)>0\}\big|}\,&\leq\, \frac{3^{n-1}}{3^{1-n}}\frac{\big|\{\tilde\varphi<0\}\cap U_\delta \big|}{\big|\{\tilde\varphi>0\}\cap U_\delta\big|}.
\end{align}
The claim thus follows if we can show the following statement:\\
For all $\tilde\varphi\in L^1(U_\delta;\{-3^{1-n},3^{n-1}\})$ with
\begin{align}
	\int_{U_\delta} \widetilde{W}(q(\frac{r(1)d}{\delta}))\tilde\varphi\,dx \,\geq\,0 \label{eq:L5-2}
\end{align}
we have
\begin{align*}
	\big|\{\tilde\varphi<0\}\cap U_\delta\big| \,&\leq\, C \,\big|\{\tilde\varphi>0\}\cap U_\delta\big|.
\end{align*}
Since $\{\tilde\varphi>0\}\cap U_\delta=U_\delta\setminus \{\tilde\varphi<0\}$ it is sufficient to prove 
\begin{align*}
	\big|\{\tilde\varphi>0\}\cap U_\delta\big| \,&\geq\, c_0
\end{align*}
for some $c_0>0$. We argue by contradiction and suppose that there is a sequence $(\tilde\varphi_k)_{k\in\N}$ satisfying condition \eqref{eq:L5-2} with 
\begin{gather*}
	\big|\{\tilde\varphi_k>0\}\cap U_\delta\big|\,\to\, 0\quad\text{ as }k\to\infty,
\end{gather*}
which immediately implies $\tilde\varphi_k\to -3^{1-n}\Chi_{U_\delta}$ in $L^1(U_\delta)$ as $k\to\infty$. But then
\begin{align*}
	\lim_{k\to\infty} \int_{U_\delta} \widetilde{W}(q(\frac{r(1)d}{\delta}))\tilde\varphi_k\,dx \,<\,0,
\end{align*}
which contradicts to \eqref{eq:L5-2}.
\end{proof}
We are now in a position to proceed with the proof of Proposition \ref{prop:dr}.

\proof[Proof of Proposition \ref{prop:dr}]

In the next calculations we fix $n=3$. The case $n=2$ follows in the same manner by a somewhat simpler calculation. We have
\begin{align*}
	\int_\Urep \frac{\widetilde{W}(u_\eps)}{\eps}\,dx &=\,
	\int_{-r(\eps)}^{r(\eps)}\int_{\{\dist(\cdot,\partial E)=s\}} \frac{\widetilde{W}(q_\eps(s))}{\eps}\,\dH^{2}ds\\
	&=\, \int_{-r(\eps)}^{r(\eps)}\frac{\widetilde{W}(q_\eps(t))}{\eps}\int_{\partial E}J_{2}\Psi(t,z) \,\dH^{2}(z)dt
	\\
	&=\,\int_{-r(\eps)}^{r(\eps)}\frac{\widetilde{W}(q_\eps(t))}{\eps}\int_{\partial E}\Pi_{i=1}^{2}(1-t\kappa_i(z)) \,\dH^{2}(z)dt\\
	&=\,\int_{-r(\eps)}^{r(\eps)}\frac{\widetilde{W}(q_\eps(t))}{\eps}\int_{\partial E}(1-tH(z)+t^2K(z))\,\dH^{2}(z)dt
\\
	&=\,\mathcal{H}^2(\partial E)\int_{-r(1)}^{r(1)}\widetilde{W}(q(r))\,dr+\int_{-r(\eps)}^{r(\eps)}\frac{\widetilde{W}(q_\eps(t))}{\eps}\int_{\partial E}t^2K(z)\,\dH^{2}(z)dt
\\
	&\leq\,\mathcal{H}^2(\partial E)\int_{-r(1)}^{r(1)}\widetilde{W}(q(r))\,dr+\int_{-r(\eps)}^{r(\eps)}\frac{t^2}{\eps}\Big\vert\int_{\partial E}K(z)\,\dH^{2}(z)\Big\vert dt\\
	&=\,\mathcal{H}^2(\partial E)\int_{-r(1)}^{r(1)}\widetilde{W}(q(r))\,dr+\frac{2r(1)\eps^2}{3}|\int_{\partial E}K\,d\Ha^2|.
\end{align*}
We therefore obtain with 
$\widetilde c_0:=\int_{-r(1)}^{r(1)}\widetilde{W}(q(r))\,dr$,
$$
 \tilde{\A}_{u_\eps,\eps}(1)=\int_\Urep\frac{\widetilde{W}(u_\eps)}{\eps}\,dx\,=\, {\widetilde c_0}S+O(\eps^2),
$$
which proves the second equality in \eqref{eq:prop3}.

We now proceed by contradiction. Suppose we can find $\varphi\in BV(\Urep;\,\{-1,1\})$ such that 
\begin{align}
\tilde{\A}_{u_\eps,\eps}(\varphi)<\tilde{\A}_{u_\eps,\eps}(1)=\widetilde c_0 S+O(\eps^2), \label{eq:nabla-small}
\end{align}
we then have
\begin{align}\label{eq:chloe}
\frac{1}{\widetilde c_0 \eps^\frac{1}{2}}\int_\Urep\frac{\widetilde W(\overline{\lambda}u_\eps)}{\eps}\,\mathrm{d}\vert\nabla\varphi\vert<\widetilde c_0 S+O(\eps^2).
\end{align}
Without loss of generality we can suppose that 
\begin{gather}
	\int_\Urep \widetilde W(u)\varphi\,dx\geq 0. \label{eq:pos-W}
\end{gather}
With $r(1)=q^{-1}(1-\lambda)$ and $r(\eps) = \eps r(1)$ as above we have 
$$
\widetilde W(u_\eps)\leq\chi_{\{\vert u_\eps\vert\le 1-\lambda\}}=\chi_{\mathcal U_{\eps r(1)}}\,\le\,\widetilde W(\overline{\lambda}u_\eps).
$$
We now set $\psi:=\frac{1-\varphi}{2}$, $\psi:\Omega\,\to\, \{0,1\}$ and observe that by \eqref{eq:pos-W} and Lemma \ref{lem:compl} 
\begin{align*}
	\mathcal L^n(\mathcal U_{\eps r(1)}\cap\{\psi > 0\}) \,\leq\,  C\mathcal L^n(\mathcal U_{\eps r(1)}\cap\{\psi=0\}).
\end{align*}
By this estimate and Lemma \ref{lem:balotelli} we then have
\begin{align*}
	\int_\Urep\widetilde{W}(u_\eps)\psi\,dx \leq &
	\mathcal L^n(\mathcal U_{\eps r(1)}\cap\{\psi>0\})\\
	\leq\, &\max\{1,C\} \min\{\mathcal L^n(\mathcal U_{\eps r(1)}\cap\{\psi>0\}),\mathcal L^n(\mathcal U_{\eps r(1)}\cap\{\psi=0\})\} \\
	\leq\,& C\eps^{-\frac{1}{n-1}}\mathcal H^{n-1}(\mathcal U_{\eps r(1)}\cap \partial^*\{\psi>0\})
	\\
	\leq & C \eps^{-\frac{1}{n-1}}\Big(\int_\Urep\widetilde W(\overline{\lambda}u_\eps)\,d\vert\nabla\psi\vert\Big)^{\frac{n}{n-1}},
\end{align*}
and hence, taking into account of \eqref{eq:chloe} and noting that $\psi \equiv 1$ is excluded by~\eqref{eq:pos-W},
\begin{gather}
	\int_{\Urep}\frac{\widetilde W(u_\eps)}{
\eps}\psi\,dx\leq C\Big(	\int_\Urep \frac{\widetilde W(\overline{\lambda}u_\eps)}{\eps}\,\mathrm{d}\vert\nabla\psi\vert\Big)^{\frac{n}{n-1}}
\leq C \eps^{\frac{\frac{1}{2}}{n-1}}S^{\frac{1}{n-1}}\int_\Urep \frac{\widetilde W(\overline{\lambda}u_\eps)}{\eps}\,\mathrm{d}\vert\nabla\psi\vert. \label{eq:est-Wpsi}
\end{gather}
But then we can find $\eps_0:=\eps_0(S,n,\partial E)$ such that for every $\eps<\eps_0$ if $\psi$ is not constant  zero (and hence $\varphi=1$ a.e.) we have
\begin{align*}
	0>&\;2\int_\Urep\frac{1}{\eps}\widetilde{W}(u)\psi\,dx-\frac{2}{\eps^\frac{1}{2}}\int_\Urep
	\frac{\widetilde W(\overline{\lambda}u_\eps)}{\eps}\,\mathrm{d}\vert\nabla\psi\vert
	\\
	=&\;\tilde{\A}_{u_\eps,\eps}(1)-\tilde{\A}_{u_\eps,\eps}(\varphi),
\end{align*}
which contradicts the minimality of $\varphi$. Therefore $\varphi=1$ is optimal and the conclusion follows.
\endproof
%%%%%%%%%%%%%%%%%%%%%%%%%%%%%%%%%%%%%%%%%%%%%%%%%%

The functional $\tilde{\A}_{u_\eps,\eps}$ as written above requires its argument $\phi$ to be a function taking only values in $\{-1,1\}$. This makes it not suitable for computation. We therefore return to the original diffuse interface formulation. In the following, we will estimate the infimum of the diffuse functional $\A_{u_\eps, \eps}$ from below by the sharp interface version $\tilde{\A}_{u_\eps,\eps}$
\begin{prop}
We have, for $\eps$ small enough and $u_\eps$ being the optimal profile construction used above,  that  $\frac{1}{\eps^{1/2-\sigma/2}}\left( \inf_{\phi \in W^{1,2}(\Omega, \R)}A_{u_\eps,\epsilon}(\phi)  - \int_\Omega \frac{1}{\eps}  \widetilde{W}(u_\eps)\right)^2 \to 0$ as $\eps \to 0$. 
\end{prop}
\begin{proof}
We will show the statement by reducing it to the case of the function $\phi$ only admitting discrete values. We set again $\Urep := \{\widetilde{W}(u_\eps) >0 \}$, note  that $\widetilde{W}(\overline{\lambda} u) = 1$ on $\Urep$ and calculate for any $\phi \in W^{1,2}(\Omega)$ that
\begin{align*}
& \quad\int_\Urep  \left(\frac{8}{\eps^{\frac{3}{2}}} +\eps \right)W(\phi) \,\dx +
\int_\Urep \left(\frac{8}{\eps^{\frac{3}{2}}} \widetilde{W}(\overline{\lambda}u_\eps)+\eps \right)  \abs{\nabla \phi}^2 \,\dx \\
&\ge \int_{\Urep} \frac{8}{\eps^{\frac{3}{2}}}\sqrt{W(\phi)} \,\abs{\nabla \phi} \,\dx\\
&\ge \int_{\Urep \cap \{\phi\in (-1/2, 1/2] \} } \frac{1}{\eps^{\frac{3}{2}}} \,\abs{\nabla \phi} \,\dx.
\end{align*}
We now fix $s_0 \in (-1/2,1/2]$ such that 
\begin{align*}
\int_{\Urep \cap \{\phi \in (-1/2, 1/2]\}} \,\abs{\nabla \phi} \,\dx
&= \int_{-1/2}^{1/2} \int_{\partial \{ \phi < s\} \cap \Urep} d\Ha^{n-1}(x) \,\mathrm{d}s \\
&\ge \int_{\partial \{ \phi < s_0\} \cap \Urep}  d\Ha^{n-1}(x)
\end{align*}
and set $\tilde{\phi} := 2\chi_{\{\phi \ge s_0\}} -1$. It follows that
\begin{align}
\tilde{\A}_{u_\eps,\eps}(\tilde{\phi}) &= \abs{\int_\Urep  \frac{1}{\eps} \widetilde{W}(u_\eps) \tilde{\phi} \,\dx} + \int_\Urep \frac{\widetilde{W}(\overline{\lambda}u_\eps)}{\eps^{\frac{3}{2}}} \,d\!\abs{\nabla \tilde{\phi}}   \nonumber \\
&\le \abs{\int_\Omega  \frac{1}{\eps} \widetilde{W}(u_\eps) \tilde{\phi} \,\dx} + 
\int_\Omega \left( \frac{8}{\eps^{\frac{3}{2}}}+\eps\right) W(\phi) \,\dx +
\int_\Omega \left(\frac{8 \widetilde{W}(\overline{\lambda}u_\eps)}{\eps^{\frac{3}{2}}} +\eps \right) \abs{\nabla \phi}^2 \,\dx. \label{eq:grad_phi_est}
\end{align}

In the final step, we need to estimate the potential gain in the first term by switching from $\phi$ to $\tilde{\phi}$ against the remaining term $\frac{1}{\eps^{\frac{3}{2}}}W(\phi)$. So we calculate again assuming without loss of generality that $\int_\Urep  \widetilde{W}(u_\eps) \phi \ge 0$. We have
$$
\int_\Urep \left( \frac{1}{\eps} \widetilde{W}(u_\eps)\phi + \frac{1}{\eps^{\frac{3}{2}}}  W(\phi) \right)\dx \ge \int_\Urep \frac{1}{\eps} \widetilde{W}(u_\eps) \left( \phi+ \frac{1}{\eps^{\frac{1}{2}}}  W(\phi) \right) \,\dx
$$
and split 
\begin{align*}
\Urep &= (\{\phi \ge s_0 \} \cup \{\phi < s_0 \} ) \cap \Urep\\
&=  (\{\tilde{\phi} = 1 \} \cup \{\tilde{\phi }  = -1 \} ) \cap \Urep.
\end{align*}
Now, on the set $\{\phi \ge s_0 \}$, we estimate the `bad set', where we the term becomes smaller when switching to the discrete $\tilde{\phi}$,
\begin{align*}
& \phi + \frac{1}{\eps^\frac{1}{2}} W(\phi) < \tilde{\phi} \\
\iff & (1-\phi)^2(1+\phi)^2 < 4 \eps^\frac{1}{2}  (1-\phi) \\
\Longrightarrow & 1-\phi >0, \quad 1-\phi < \frac{4\eps^\frac{1}{2} }{(1+s_0)^2}.
\end{align*}
Thus, the `bad set' is contained in the set $\{ 1-  \frac{4\eps^\frac{1}{2} }{(1+s_0)^2} < \phi < 1\}$ and we have
\begin{align}
&\;\int_{ \{\phi \ge s_0\} \cap \Urep} \frac{1}{\eps} \widetilde{W}(u_\eps) \left( \phi+ \frac{1}{\eps^\frac{1}{2} } W(\phi) \right) \,\dx \nonumber \\
\ge &\; \int_{ \{\phi\ge s_0\} \cap \Urep} \frac{1}{\eps} \widetilde{W}(u_\eps) \tilde{\phi}  \,\dx 
- \int_{\{ 1-  \frac{4\eps^\frac{1}{2} }{(1+s_0)^2} < \phi < 1\} \cap \Urep} \frac{1}{\eps}\widetilde{W}(u_\eps) (1-\phi) \,\dx \nonumber \\
\ge &\;  \int_{ \{\phi\ge s_0\} \cap \Urep} \frac{1}{\eps} \widetilde{W}(u_\eps) \tilde{\phi}  \,\dx -
 \frac{4\eps^\frac{1}{2} }{(1+s_0)^2} \left( \tilde{c}_0 S + O(\eps^2) \right).  \label{eq:phi_est_1}
\end{align}
On $\{\phi < s_0\}\cap \Urep$ we make a similar estimate. Note that here
\begin{align*}
& \phi + \frac{1}{\eps^\frac{1}{2}} W(\phi) < \tilde{\phi} \\
\iff & (1-\phi)^2(1+\phi)^2 < 4 \eps^\frac{1}{2}  (-1-\phi) \\
\Longrightarrow & -1-\phi >0, \quad -1-\phi < \eps^\frac{1}{2} .
\end{align*}
This implies that
\begin{align}
&\;\int_{ \{\phi< s_0\} \cap \Urep} \frac{1}{\eps} \widetilde{W}(u_\eps) \left( \phi+ \frac{1}{\eps^\frac{1}{2} } W(\phi) \right) \,\dx  \nonumber \\
\ge &\;  \int_{ \{\phi <s_0\} \cap \Urep} \frac{1}{\eps} \widetilde{W}(u_\eps) \tilde{\phi}  \,\dx -
\int_{ \{ -1-\eps^\frac{1}{2}  < \phi < -1 \}\cap \Urep} \frac{1}{\eps} \widetilde{W}(u_\eps) (-1 -\phi) \dx \nonumber \\
\ge &\;  \int_{ \{\phi <s_0\}\cap \Urep } \frac{1}{\eps} \widetilde{W}(u_\eps) \tilde{\phi}  \,\dx  - \eps^\frac{1}{2}  \left( \tilde{c}_0 S + O(\eps^2) \right). \label{eq:phi_est_2}
\end{align}
Adding~\eqref{eq:phi_est_1} and~\eqref{eq:phi_est_2} we thus get 
\begin{align}
\int_\Urep \left( \frac{1}{\eps} \widetilde{W}(u_\eps)\phi + \widetilde{W}(\overline{\lambda}u_\eps)\frac{1}{\eps^{\frac{3}{2}}}  W(\phi) \right)\dx \ge   \int_\Urep \frac{1}{\eps} \widetilde{W}(u_\eps) \tilde{\phi}  \,\dx  - O(\eps^\frac{1}{2} ) \label{eq:phi_est_2.5}
\end{align}

It remains to show that we also have
$$
\int_\Urep \left( \frac{1}{\eps} \widetilde{W}(u_\eps)\phi + \widetilde{W}(\overline{\lambda}u_\eps)\frac{1}{\eps^{\frac{3}{2}}}  W(\phi) \right)\dx \ge  - \int_\Urep \frac{1}{\eps} \widetilde{W}(u_\eps) \tilde{\phi}  \,\dx  - O(\eps^\frac{1}{2} ),
$$
since we need to estimate the absolute value of the $\tilde{\phi}$-term and can only assume without loss of generality that $\int_\Urep \frac{1}{\eps} \widetilde{W}(u_\eps)\phi \,\dx \ge 0$. Note however that
$$
 \int_\Urep \frac{1}{\eps} \widetilde{W}(u_\eps) \tilde{\phi}  \,\dx
= \int_{ \{\phi\ge s_0\} \cap \Urep} \frac{-1}{\eps}  \widetilde{W}(u_\eps) \,\dx + 
 \int_{ \{\phi< s_0\} \cap \Urep} \frac{1}{\eps}  \widetilde{W}(u_\eps) \,\dx.
$$
We can thus estimate
\begin{align}
&\;\int_{ \{\phi\ge s_0\} \cap \Urep} \frac{1}{\eps} \widetilde{W}(u_\eps) \left( \phi+ \frac{1}{\eps^\frac{1}{2} } W(\phi) +\tilde{\phi }\right) \,\dx \nonumber \\
= & \;\int_{ \{\phi\ge s_0\} \cap \Urep} \frac{1}{\eps} \widetilde{W}(u_\eps) \left( \phi+ \frac{1}{\eps^\frac{1}{2} } W(\phi)+1\right) \,\dx \nonumber \\
\ge &\; 2 \int_{ \{\phi\ge s_0\} \cap \Urep} \frac{1}{\eps} \widetilde{W}(u_\eps) \phi\,\dx +
 \int_{ \{\phi\ge s_0\} \cap \{1<\phi<1+4\eps^\frac{3}{2} \} \cap \Urep} \frac{1}{\eps} \widetilde{W}(u_\eps) (1-\phi)\,\dx \nonumber \\
\ge &\; 2 \int_{ \{\phi\ge s_0\} \cap \Urep} \frac{1}{\eps} \widetilde{W}(u_\eps) \phi\,\dx - 4\eps^\frac{1}{2}   \left( \tilde{c}_0 S + O(\eps^2) \right), \label{eq:phi_est_3}
\end{align}
where we have used that
\begin{align*}
& \phi + 1+ \frac{1}{\eps^\frac{1}{2}} W(\phi) < 2\phi \\
\iff & (1-\phi)^2(1+\phi)^2 < 4 \eps^\frac{1}{2}  (\phi-1) \\
\Longrightarrow & \,\phi-1 >0, \quad \phi-1 < 4\eps^\frac{1}{2} .
\end{align*}
Similarly, we have
\begin{align}
&\;\int_{ \{\phi< s_0\} \cap \Urep} \frac{1}{\eps} \widetilde{W}(u_\eps) \left( \phi+ \frac{1}{\eps^\frac{1}{2} } W(\phi) +\tilde{\phi }\right) \,\dx \nonumber \\
= &\; \int_{ \{\phi< s_0\} \cap \Urep} \frac{1}{\eps} \widetilde{W}(u_\eps) \left( \phi+ \frac{1}{\eps^\frac{1}{2} } W(\phi)-1\right) \,\dx \nonumber \\
\ge &\; 2 \int_{ \{\phi<s_0\} \cap \Urep} \frac{1}{\eps} \widetilde{W}(u_\eps) \phi\,\dx -
 \int_{ \{\phi< s_0\} \cap \{-1<\phi < \frac{4\eps^\frac{1}{2}}{(1+s_0)^2} -1\} \cap \Urep} \frac{1}{\eps} \widetilde{W}(u_\eps) (\phi+1)\,\dx \nonumber \\
\ge &\; 2 \int_{ \{\phi< s_0\} \cap \Urep} \frac{1}{\eps} \widetilde{W}(u_\eps) \phi\,\dx -  \frac{4\eps^\frac{1}{2}}{(1+s_0)^2}   \left( \tilde{c}_0 S + O(\eps^2) \right),  \label{eq:phi_est_4}
\end{align}
where we have used that
\begin{align*}
& \phi - 1+ \frac{1}{\eps^\frac{1}{2}} W(\phi) < 2\phi \\
\iff & (1-\phi)^2(1+\phi)^2 < 4 \eps^\frac{1}{2}  (\phi+1) \\
\Longrightarrow & \,\phi+1 >0, \quad \phi+1 < \frac{4\eps^\frac{1}{2} }{(1+s_0)^2}.
\end{align*}
Adding now~\eqref{eq:phi_est_3} and~\eqref{eq:phi_est_4} we get
\begin{align}
&\;\int_{ \Urep } \frac{1}{\eps} \widetilde{W}(u_\eps) \left( \phi+ \frac{1}{\eps^\frac{1}{2} } W(\phi) \right) \,\dx \nonumber \\
\ge &\; 2 \int_{ \{\phi< s_0\} \cap \Urep} \frac{1}{\eps} \widetilde{W}(u_\eps) \phi\,\dx -\int_{ \cap \Urep } \frac{1}{\eps} \widetilde{W}(u_\eps) \tilde{\phi}\,\dx - O(\eps^\frac{1}{2} )\nonumber \\
 \ge &\; -\int_{ \Urep} \frac{1}{\eps} \widetilde{W}(u_\eps) \tilde{\phi}\,\dx - O(\eps^\frac{1}{2} ). \label{eq:phi_est_5}
\end{align}
Equations~\eqref{eq:phi_est_5} and \eqref{eq:phi_est_2.5} together yield
\begin{align}
\int_\Omega \left( \frac{1}{\eps} \widetilde{W}(u_\eps)\phi +  \frac{1}{\eps^{\frac{3}{2}}} \widetilde{W}(\overline{\lambda}u_\eps) W(\phi) \right)\dx
 \ge  \abs{\int_{ \Omega } \frac{1}{\eps} \widetilde{W}(u_\eps) \tilde{\phi}\,\dx} - O(\eps^\frac{1}{2} ). \label{eq:phi_est_final}
\end{align}

Adding \eqref{eq:grad_phi_est} and \eqref{eq:phi_est_final}, and using the result from Proposition~\ref{prop:dr}, we get that
$$
\int_{\Omega} \frac{1}{\eps}\widetilde{W}(u_\eps)\,\dx - O(\sqrt{\eps}) \le \tilde{\A}_{u_\eps,\epsilon}(\tilde{\phi})- O(\sqrt{\eps}) \le  \A_{u_\eps,\eps}(\phi) \le \int_{\Omega} \frac{1}{\eps}\widetilde{W}(u_\eps)\,\dx
$$
for any $\phi \in W^{1,2}(\Omega)$ and $\eps$ sufficiently small. The upper bound here is nothing but the trivial estimate~\eqref{eq:A_upper_bd}. This yields the desired convergence.
\end{proof}
The preceding arguments show that the sequence $u_\eps$ is a suitable recovery sequence. This finishes the proof of Theorem~\ref{thm:main}.
%===================================
% Liminf
%===================================
\section{Lim\,inf inequality.} 
Most of the assertions of Theorem \ref{thm:lower} follow immediately from \cite{RoeSc06}. The difficult part is to prove that $\mu$ is concentrated on a connected structure in the sense of the corresponding statement in Theorem \ref{thm:lower}. 
In the following we assume that \eqref{eq:pf-conf}, \eqref{eq:liminf3}-\eqref{eq:liminf2} hold. Without loss of generality we can pass to a subsequence that realizes the $\liminf$ in \eqref{eq:liminf9} and for which
\begin{align*}
	\alpha_\eps \,&\to\, \alpha\quad\text{ as Radon measures},\\
	\alpha_\eps \,&:=\, \frac{1}{\eps}\Big(\eps\Delta u_\eps-\frac{W'(u_\eps)}{\eps}\Big)^2\mathcal L^n{\lfloor\Omega},
\end{align*}
for some Radon measure $\alpha$ on $\Omega$.

In general we cannot expect that $u_\eps$ converges to $\pm 1$ uniformly on compact sets separated from the support of $\mu$. However, we obtain the following statement.
\begin{lemma} \label{rect.lower}
There exists some universal $\theta_0>0$ with the following property: if $x_0\in \Omega\setminus \spt({\mu})$ satisfies for some $r_0>0$ with $B(x_0,2r_0)\subset\subset \Omega$
\begin{align}
	 \limsup_{\eps\to 0} \int_{B(x_0,r_0)} \eps^{-\frac{3}{2}}\widetilde{W}(\overline{\lambda}u_\eps)\,d\LL^n \,>\,0 \label{eq:mu-bar-big}
\end{align}
then $\alpha({x_0})>\theta_0$ holds. In particular, there are only finitely many points in $\Omega\setminus \spt({\mu})$ such that \eqref{eq:mu-bar-big} holds.
\end{lemma}
\begin{proof}
We closely follow \cite[Lemma 4.8]{RoeSc06}. Without loss of generality we may assume $x_0=0$, $r_0<1$, and we may pass to a subsequence $\eps\to 0$ that realizes the $\limsup$ in \eqref{eq:mu-bar-big}. Choose $\beta=\beta(r_0)>0$ such that
\begin{eqnarray*}
	3\left(\frac{r_0}{4} \right)^{1-\beta} &\leq& r_0.
\end{eqnarray*}
For $x \in B(0,r/2), 0 < r \leq r_0 \leq 1$,
we have $B(x,r) \subset B(0,3 r_0/2) \subset \subset \Omega$
and obtain from \cite[Proposition 3.6, Proposition 4.6]{RoeSc06} that for $\eps\leq s\leq r/4$
\begin{align}
	(r/4)^{1-n} \mu_\varepsilon(B(x,r/4))
	\,\geq \, & s^{1-n} \mu_\varepsilon(B(x,s))
	- C \int \limits_s^{r/4} \varrho^{-1+\gamma} \varrho^{1-n}
	\mu_\varepsilon(B(x,2\varrho)) \,d\varrho \notag\\
	&\quad - C_\beta \varepsilon^2 \int \limits_s^{r/4}
	\varrho^{-M \gamma - n}
	\alpha_\varepsilon(B(x,3 \varrho^{1-\beta})) \,d\varrho
	- C_\beta(r_0) \varepsilon^{\gamma} \notag\\
	&\qquad - C \alpha_\varepsilon(B(x,r/4)), \label{rect.lower.aux}
\end{align}
for all $\beta>0$, a universal $M>0$, and all $0 < \gamma < 1/M $.

Next we seek a point $x \in B(0,r/2)$ satisfying
\begin{align}
	\varepsilon^{1-n} \mu_\varepsilon(B(x,\eps))
	\,\geq\, 2 \bar{\theta}_0 \,>\, \bar{\theta}_0
	\,\geq\, C_\beta \varepsilon^2 \int \limits_\varepsilon^{r/4}
	\varrho^{-M \gamma-n}
	\alpha_\varepsilon(B(x,3\varrho^{1-\beta})) \,d\varrho  \label{rect.lower.cond}
\end{align}
for some universal $\bar{\theta}_0 > 0$.
We consider $x \in B(0,r/2)$ with $|u_\varepsilon(x)|
\leq 1 - \tau$ for some $0 < \tau < 1-\overline{\lambda}^{-1}(1-\lambda)$\footnote{That is, such that $\widetilde{W}(1-\tau)$ still vanishes.}.
If $\varepsilon^{1-n} \mu_\varepsilon(B(x,\eps))
\leq 1$, we deduce
\begin{align*}
	\varepsilon^{-n} \int \limits_{B(x,\eps)} u_\varepsilon^4
	\,\leq\, C \Big( 1 + \int \limits_{B(x,\eps)}
	\varepsilon^{-n} W(u_\varepsilon) \Big) \,\leq\, C.
\end{align*}
As for $n\leq 3$
\begin{align*}
	\parallel \varepsilon v_\varepsilon(x + \varepsilon \cdot )
	\parallel_{L^2(B(0,1))}^2
	\,\leq\, C \varepsilon^{2-n} \int \limits_{B(x,\eps)}
	v_\varepsilon^2 \,\leq\, C \alpha_\varepsilon(B(x,\eps))
	\,\leq\, C,
\end{align*}
it follows from standard elliptic estimates that
\begin{align*}
	\parallel u_\varepsilon(x + \varepsilon \cdot)
	\parallel_{C^{0,1/2}(B(0,1/2))}
	\,\leq\, C \parallel u_\varepsilon(x + \varepsilon \cdot)
	\parallel_{W^{2,2}(B(0,1/2))} \,\leq\, C,
\end{align*}
hence
\begin{align*}
	|u_\varepsilon| \,\leq\, 1 - \tau/2
	\quad \text{ on } B(x,c_0\tau^2\eps)
\end{align*}
for $c_0 \ll 1$ small enough and
\begin{align*}
	\varepsilon^{1-n} \mu_\varepsilon(B(x,\eps))
	\,\geq\, \varepsilon^{-n} \int \limits_{B(x,c_0\tau^2\eps)}
	W(u_\varepsilon) \,\geq\, c_0 \tau^{2n+2} \,:=\, 2 \bar{\theta}_0 \,>\, 0.
\end{align*}
For $c_0 \ll 1$ this is also true in case
$\varepsilon^{1-n} \mu_\varepsilon(B(x,\eps)) \geq 1$,
and we get
\begin{align} \label{rect.lower.cond1}
	\varepsilon^{1-n} \mu_\varepsilon(B(x,\eps))
	\,\geq\, 2 \bar{\theta}_0
	\quad \text{ for } x \in B(0,r/2)
	\cap \{|u_\varepsilon| \leq 1 - \tau\}.
\end{align}
As $\tau <1-\overline{\lambda}^{-1}(1-\lambda)$ we deduce from the definition of $\widetilde{W}$ that
\begin{align*}
	\int_{B(0,r/2)} \eps^{-\frac{3}{2}}\widetilde{W}(\overline{\lambda}u_\eps)\,dx 
	\,=\,  \int_{B(0,r/2)\cap \{|u_\eps|\leq 1-\tau\}} \eps^{-\frac{3}{2}}\widetilde{W}(\overline{\lambda}u_\eps)\,dx,
\end{align*}
hence, by \eqref{eq:mu-bar-big} for all $0< \tau < 1-\overline{\lambda}^{-1}(1-\lambda)$
\begin{align}
	\liminf \limits_{\varepsilon \rightarrow 0}
	\varepsilon^{-\frac{3}{2}} \LL^n \Big( B(0,r/2) \cap
	\{|u_\varepsilon| \leq 1 - \tau\} \Big)
	\geq\,&
	\liminf \limits_{\varepsilon \rightarrow 0}
	 \int \limits_{B(0,r/2)
	\cap \{|u_\varepsilon| \leq 1 - \tau\}}
	\eps^{-\frac{3}{2}}\widetilde{W}(\overline{\lambda}u_\eps) \notag\\
	= \,&\liminf \limits_{\varepsilon \rightarrow 0}
	 \int \limits_{B(0,r/2)}
	\eps^{-\frac{3}{2}}\widetilde{W}(\overline{\lambda}u_\eps)\,>\, 0. \label{rect.lower.aux2}
\end{align}
To estimate the integral on the right-hand side of \eqref{rect.lower.cond},
we define for $0 < \varrho \leq r_0$ the convolution
\begin{align*}
	w_{\varepsilon,\varrho}(x)
	:= \varrho^{-n} \Big( \chi_{B(0,\varrho)}
	* \frac{1}{\varepsilon} v_\varepsilon^2 \Big)(x)
	= \varrho^{-n} \alpha_\varepsilon(B(x,\varrho))
\end{align*}
and see $w_{\varepsilon,\varrho} \in L^1(B(0,r_0/2))$ with
\begin{align*}
	\parallel w_{\varepsilon,\varrho}
	\parallel_{L^1(B(0,r_0/2))}
	\,\leq\, \int \limits_{B(0,r_0/2+\varrho)}
	\frac{1}{\varepsilon} v_\varepsilon^2
	\,\leq\, \alpha_\varepsilon(B(0,3 r_0/2)) < \infty.
\end{align*}
Putting $w_\varepsilon := \int_0^{r_0}
w_{\varepsilon,\varrho} \,d\varrho$, we see
\begin{align*}
	\parallel w_\varepsilon \parallel_{L^1(B(0,r_0/2))}
	\leq r_0 \alpha_\varepsilon(B(0,3 r_0/2)) < \infty
\end{align*}
and calculate
\begin{align*}
	&\int \limits_\varepsilon^{r/4} \varrho^{-M \gamma-n}
	\alpha_\varepsilon(B(x,3 \varrho^{1-\beta})) \,d\varrho \\
	 =\,&
	 \int \limits_{3 \varepsilon^{1-\beta}}^{3 (r/4)^{1-\beta}}
	(t/3)^{(-M \gamma-n)/(1-\beta)}
	\alpha_\varepsilon(B(x,t))
	t^{1/(1-\beta) - 1} 3^{-1/(1-\beta)}
	(1-\beta)^{-1} \,dt \\
	\leq\,&
	C \int \limits_{3 \varepsilon^{1-\beta}}^{3 (r/4)^{1-\beta}}
	t^{(-M \gamma-n+\beta)/(1-\beta)}
	\alpha_\varepsilon(B(x,t)) \,dt \\
	\leq\,& C \varepsilon^{(-M\gamma - (n-1) \beta)/(1-\beta)}
	\int \limits_0^{r_0}
	w_{\varepsilon,\varrho}(x) \,d\varrho
	\,=\, C  \varepsilon^{(-M\gamma - (n-1) \beta)/(1-\beta)}
	w_\varepsilon(x).
\end{align*}
Choosing $\gamma,\beta$ such that $(M\gamma + (n-1) \beta)/(1-\beta)<1/2$ we get
\begin{align*}
	&\limsup \limits_{\varepsilon \rightarrow 0}
	\varepsilon^{-3/2} \LL^n \bigg( B(0,r/2) \cap
	\Big\{ C_\beta \varepsilon^2 \int \limits_\varepsilon^{r/4}
	\varrho^{-M \gamma -n}
	\alpha_\varepsilon(B(x,3 \varrho^{1-\beta}))
	\,d\varrho \geq \bar{\theta}_0 \Big\} \bigg) \\
	\leq\,&
	 \limsup \limits_{\varepsilon \rightarrow 0}
	\varepsilon^{-3/2} C_\beta
	\varepsilon^{2 - (M\gamma + (n-1) \beta)/(1-\beta)}
	\bar{\theta}_0^{-1} \parallel w_\varepsilon
	\parallel_{L^1(B(0,r_0/2))} \\
	\leq\,&
	\limsup \limits_{\varepsilon \rightarrow 0}
	\varepsilon^{1/2 - (M\gamma + (n-1) \beta)/(1-\beta)}
	C_\beta \bar{\theta}_0^{-1}
	r_0 \alpha_\varepsilon(B(0,3 r_0/2))\, =\, 0.
\end{align*}
Combining with \eqref{rect.lower.cond1} and \eqref{rect.lower.aux2},
we see for $\varepsilon>0$ small enough that $x \in B(0,r/2)$
satisfying \eqref{rect.lower.cond} exists.

For such $x$, we claim
\begin{align} \label{rect.lower.aux3}
	(r/2)^{1-n} \mu_\varepsilon(B(x,r/2))
	\,\geq\, 2^{1-n} \bar{\theta}_0
	- C_\gamma r^\gamma
	- C_\beta(\Omega') \varepsilon^\gamma
	- C \alpha_\varepsilon(B(x,r/4)).
\end{align}
If not, we put
\begin{align*}
	s := \sup \{ \varepsilon \leq \varrho \leq r/2\ |
	\ \varrho^{1-n} \mu_\varepsilon(B(x,\varrho)) \geq 2 \bar{\theta}_0\ \}.
\end{align*}
Clearly $\varepsilon \leq s \leq r/4$, as we assume
that \eqref{rect.lower.aux3} is not satisfied, and
\begin{align*}
	s^{1-n} \mu_\varepsilon(B(x,s)) \,&\geq\, 2 \bar{\theta}_0, \\
	\varrho^{1-n} \mu_\varepsilon(B(x,\varrho))
	\,&\leq\, 2 \bar{\theta}_0
	\qquad \text{ for all } s \leq \varrho \leq r/2.
\end{align*}
Then we obtain from \eqref{rect.lower.aux}
and \eqref{rect.lower.cond}
\begin{align*}
	&2^{n-1} (r/2)^{1-n} \mu_\varepsilon(B(0,r/2))\\
	\geq\,& (r/4)^{1-n} \mu_\varepsilon(B(0,r/4))\\
	\geq \,& 2 \bar{\theta}_0 - C \int \limits_s^{r/4}
	2 \bar{\theta}_0 \varrho^{-1+\gamma} \,d\varrho
	- \bar{\theta}_0
	- C_\beta(\Omega') \varepsilon^{\gamma}
	- C \alpha_\varepsilon(B(x,r/4)) \\
	\geq\,& \bar{\theta}_0
	- C_\gamma r^\gamma
	- C_\beta(\Omega') \varepsilon^{\gamma}
	- C \alpha_\varepsilon(B(x,r/4))
\end{align*}
which yields \eqref{rect.lower.aux3}.

As $B(x,r/2) \subseteq B(0,r)$,
we get from \eqref{rect.lower.aux3}
for $\varepsilon \rightarrow 0$
\begin{align*}
	r^{1-n} \mu(\overline{B(0,r)})
	\,&\geq\, \limsup \limits_{\varepsilon \rightarrow 0}
	2^{1-n} (r/2)^{1-n} \mu_\varepsilon(B(x,r/2))\\
	&\geq\,  4^{1-n} \bar{\theta}_0
	- C_\gamma r^\gamma
	- C \alpha(\overline{B(0,r)}).
\end{align*}
Approximating $r' \nearrow r$,
we get for all $0 < r \ll 1$ that
\begin{align*}
	r^{1-n} \mu(B(0,r)) \geq \frac{1}{20}\bar{\theta}_0 - C \alpha(B(0,r)).
\end{align*}
As the left-hand side is zero for $r>0$ sufficiently small we deduce that $\alpha(\{0\})\geq \frac{1}{20 C}\bar{\theta}_0=:\theta_0$.
\end{proof}

We are now prepared to prove the lower bound estimates.
\begin{proof}[Proof of Theorem \ref{thm:lower}]
 As a direct consequence of the results proved in \cite{RoeSc06} we obtain  \eqref{eq:liminf4},  \eqref{eq:liminf6},  \eqref{eq:liminf7},% \eqref{eq:liminf7bis} 
 and \eqref{eq:liminf9}. Moreover, \eqref{eq:liminf5} follows from \eqref{eq:liminf3}, as this implies
$$
\limsup_{\eps\to 0}\vert\mathcal S_\eps(u_\eps)-S\vert=\limsup_{\eps\to 0}\vert\mu_\eps(\Omega)-S\vert\leq C\lim_{\eps\to 0}\eps^{1-\frac{\sigma}{2}}=0.
$$
To prove the connectedness of the support as formulated in the Theorem we
argue by contradiction and suppose that we can find open sets $\Omega_1,\Omega_2\subset\R^n$ such that $\overline{\Omega_1}\cap\overline{\Omega_2}=\emptyset$, and such that
$$
\mu(\Omega_i)>0~(i=1,2),\quad\mu(\R^n\setminus(\Omega_1\cup\Omega_2)=0.
$$
By Lemma \ref{rect.lower} and a simple covering argument there exists a finite set $B \subset \Omega\setminus(\overline\Omega_1\cup\overline\Omega_2)$ such that
for
\begin{gather*}
	U \,:=\,  \Omega\setminus(\overline\Omega_1\cup\overline\Omega_2\cup B)
\end{gather*}
and any $V\subset\subset U$
\begin{gather}
	\int_V \eps^{-3/2}\widetilde{W}(\overline{\lambda}u_\eps) \,\to\, 0 \, \text{ as }\eps\to 0. \label{eq:grad_pen_vanish}
\end{gather}
Since $\delta_0 := \dist(\Omega_1,\Omega_2)>0$ and since $B$ is finite there exist $0<\delta_1<\delta_2 <\delta_0$ such that
\begin{gather}
	\Omega_1 \,\subset\subset \{ \dist(\cdot,\Omega_1)<\delta_1\},\\
	V\,:=\, \{x\in \R^n \,:\, \delta_1\,<\, \dist(x,\Omega_1)\,<\, \delta_2\}\,\subset\subset \big(\Omega_2\cup B \big)^c
\end{gather}
We then choose $\phi\in C^\infty(\Omega)$ such that
\begin{align*}
	\phi \,&=\, -1 \quad\text{ on } \{ \dist(\cdot,\Omega_1)<\delta_1\},\\
	\phi \,&=\, 1 \quad\text{ on } \{ \dist(\cdot,\Omega_1)>\delta_2\}
\end{align*}
In particular $\{|\nabla\phi|\neq 0\}\,\subset\, V$. We can assume without loss of generality that $\int_\Omega  \frac{1}{\eps} \widetilde{W}(u) \phi\geq 0$ (otherwise interchange the role of $\Omega_1$ and $\Omega_2$). We then obtain
\begin{align*}
	\int_\Omega  \frac{1}{\eps} \widetilde{W}(u_\eps) \phi \,\dx -\int_\Omega  \frac{1}{\eps} \widetilde{W}(u_\eps)  \,\dx 
	\,&\leq \,  - \int_{\{ \dist(\cdot,\Omega_1)<\frac{\delta_1}{2} \}}  \frac{2}{\eps} \widetilde{W}(u_\eps)\,\dx - \int_V \frac{1}{\eps} \widetilde{W}(u_\eps)  \,\dx\\
	&=:\, - T_1 - T_2.
\end{align*}
Note that $T_2$ vanishes as $\eps \to 0$ similarly to equation~\eqref{eq:grad_pen_vanish}.
Now, by $\widetilde{W}(s)=1$ for $|s|<1-\lambda$, we get
\begin{align*}
	\frac{2}{\eps}\widetilde{W}(u_\eps)\,&\geq\, \frac{2}{\eps}\widetilde{W}(u_\eps)\Chi_{\{|u_\eps|< 1-\lambda\}} \\
	&\geq\, \frac{2}{\eps}W(u_\eps) \Chi_{\{|u_\eps|< 1-\lambda\}} \\
	&=\, \frac{2}{\eps}W(u_\eps)  - \frac{2}{\eps}W(u_\eps)  \Chi_{\{|u_\eps|\geq  1-\lambda\}}\\
	&=\, \Big(\frac{\eps}{2}|\nabla u_\eps|^2 + \frac{1}{\eps}W(u_\eps)\Big) -\Big(\frac{\eps}{2}|\nabla u_\eps|^2 - \frac{1}{\eps}W(u_\eps)\Big)  - \frac{2}{\eps}W(u_\eps)  \Chi_{\{|u_\eps|\geq  1-\lambda\}}\\
	&=:\, t_{11} - t_{12} - t_{13}.
\end{align*}
Note that $t_{12}$ is nothing but the density of the discrepancy measure, whose integral vanishes according to~\cite[ Proposition 4.9]{RoeSc06}.
Assuming now  $\lambda<\frac{1}{50}$, by Proposition~\ref{prop:small_deviations} below we have
\begin{align*}
	\int_{ \{ \dist(\cdot,\Omega_1)<\frac{\delta_1}{2} \}} t_{13}\,dx \,&\leq\, 60\lambda \int_{\{|u_\eps|\leq  1-\lambda\} \cap \{ \dist(\cdot,\Omega_1)<\delta_1 \}} \frac{\eps}{2} |\nabla u_\eps|^2 + \omega(\eps)
\end{align*}
with $\omega(\eps)\to 0$ for $\eps\to 0$.

Collecting all the terms above, we thus see that
\begin{align*}
&\int_\Omega  \frac{1}{\eps} \widetilde{W}(u_\eps) \phi \,\dx -\int_\Omega  \frac{1}{\eps} \widetilde{W}(u_\eps)  \,\dx \\
\le\;& -c_0\left(  \mu_\eps(\{ \dist(\cdot,\Omega_1)<\frac{\delta_1}{2} \}) -  60\lambda \mu_\eps\left(\{ \dist(\cdot,\Omega_1)<\delta_1 \}\right)  \right) + \omega(\eps),
\end{align*}
with some function $\omega(\eps) \to 0$ as $\eps \to 0$. Together with~\eqref{eq:grad_pen_vanish}, we see that in the limit of $\eps \to 0$, we get for any choice of $\lambda < \frac{1}{60}$ that
$$
\lim_{\eps \to 0} \left( \inf_{\phi\in W^{1,2}(\Omega)} \A_{u_\eps,\eps}(\phi)- \int_\Omega \frac{1}{\eps} \widetilde{W}(u_\eps) \right)^2 \ge c\mu(\Omega_1)^2 >0
$$
and the respective term in the energy thus becomes infinite.
\end{proof}

\begin{prop}
\label{prop:small_deviations}
Consider $\Omega'' \subset\subset \Omega' \subset\subset \Omega$. For $u_\eps$ as in Theorem~\ref{thm:lower}, we have
$$
 \int_{ \{ \abs{u_\eps} \ge 1-\lambda  \} \cap \Omega''} \frac{1}{\eps} W(u_\eps) \le  15\lambda \int_{  \{\abs{u_\eps} \le 1-\lambda \} \cap \Omega'  } \eps \abs{\nabla u_\eps}^2 + \omega(\eps),
$$
where $\omega(\eps) \to 0$ as $\eps \to 0$, as long as $\lambda$ is below $1/50$\footnote{The exact numerical values can of course be improved somewhat, but the estimates given here are of a magnitude that certainly allows for an implementation in a computer simulation.}.
\end{prop}
\begin{proof}
We follow the arguments in the proof of~\cite[Proposition 3.4]{RoeSc06}. Set $g(t) := W'(t)$ for $\abs{t} \ge 1-\lambda$, $g(t) := 0$ for $\abs{t} \le t_0 = \frac{\sqrt{3}}{3}$ and interpolate $g$ linearly on the two intervals in between. Fix $\eta \in C_0^1(\Omega')$, $0 \le \eta \le 1$, $\eta =1$ on $\Omega''$ with $\abs{\nabla\eta}\le C$. Taking $\nu_\eps = -\eps \Delta u_\eps + \frac{1}{\eps}W'(u_\eps)$, we calculate
\begin{align}
\int \nu_\eps g(u_\eps) \eta^2 &= \int \eps g'(u_\eps) \abs{\nabla u_\eps}^2 \eta^2 + 2\int \eps \nabla u_\eps g(u_\eps) \eta \nabla \eta \nonumber \\
& \quad + \int \frac{1}{\eps} W'(u_\eps)g(u_\eps) \eta^2. \label{eq:p3.4_1}
\end{align}
Using Young's inequality and noting that $\abs{g} \le \abs{W'}$, we get
\begin{equation}
\abs{\int \nu_\eps g(u_\eps) \eta^2 } \le \frac{\eps}{2} \int_{\Omega'} \nu_\eps^2 + \frac{1}{2\eps} \int W'(u_\eps) g(u_\eps) \eta^2.  \label{eq:p3.4_2}
\end{equation}
for any $M>0$.

Using $\abs{g(t)} \le 2\lambda$ for $|t|\le 1-\lambda$, we calculate
\begin{align}
\abs{2\int \eps \nabla u_\eps g(u_\eps) \eta \nabla \eta} &\le 4\lambda \int_{\{ \abs{u_\eps} \le 1-\lambda \}} \eps \abs{\nabla u_\eps} \eta \abs{\nabla \eta} + \abs{2 \int_{\{ \abs{u_\eps} \ge 1-\lambda \} }  \eps \nabla u_\eps W'(u_\eps) \eta \nabla \eta } \nonumber \\
&\le  \frac{4\lambda}{M} \int_{\{ \abs{u_\eps} \le 1-\lambda \} \cap \Omega'} \eps \abs{\nabla u_\eps}^2 + MC\eps \mathcal{L}^n(\Omega') + \tau \int_{\{ \abs{u_\eps} \ge 1-\lambda \}} \eps\abs{\nabla u_\eps}^2 \eta^2 \nonumber \\
& \quad + C\eps \frac{1}{\tau} \int_{\{ \abs{u_\eps} \ge 1-\lambda \} \cap \Omega'} W'(u_\eps)^2 \label{eq:p3.4_3}
\end{align}
for any $\tau>0$ and any $M>0$. Since $g'(t) \ge c_1 > 0$ for $\abs{t} \ge 1-\lambda$ and  $\abs{g'(t)}\le\frac{2\lambda}{1-\lambda-t_0}:=c_2$ for $|t|\le 1-\lambda$, we get from equations~\eqref{eq:p3.4_1},~\eqref{eq:p3.4_2} and~\eqref{eq:p3.4_3} that
\begin{align*}
&c_1 \int_{  \{\abs{u_\eps} \ge 1-\lambda \} } \eps \abs{\nabla u_\eps}^2 \eta^2 + \frac{1}{2\eps} \int W'(u_\eps)g(u_\eps) \eta^2 \\
&\le \left(c_2+\frac{4\lambda}{M} \right) \int_{  \{\abs{u_\eps} \le 1-\lambda \} \cap \Omega'  } \eps \abs{\nabla u_\eps}^2 + \tau
\int_{ \{\abs{u_\eps} \ge 1-\lambda \} } \eps \abs{\nabla u_\eps}^2 \eta^2 + \frac{\eps}{2} \int_{\Omega'} \nu_\eps^2 \\
&\quad + \eps\left( MC + \frac{C}{\tau} \right) \mathcal{L}^n(\Omega') + C\eps\frac{1}{\tau} \int_{ \{\abs{u_\eps} \ge 1 \} \cap \Omega' } W'(u_\eps)^2.
\end{align*}	
Picking now $\tau$ sufficiently small, $M$ sufficiently large, and noting that $W'g \ge W$ on $\{ \abs{u_\eps} \ge 1-\lambda \}$, the desired result follows using uniform boundedness of $||u_\eps||_{L^p}$~\cite[Proposition 3.6]{RoeSc06}.
\end{proof}

\bibliographystyle{plain}
\bibliography{pftc}

\end{document}